\documentclass[reqno]{amsart}

\usepackage{amssymb} 
\usepackage{amsmath, amsthm} 

\usepackage[top=3cm, bottom=3cm, left=3cm, right=3cm]{geometry}

\usepackage{color}

\usepackage{enumitem} 
\usepackage{multirow} 

\usepackage{url} 

\usepackage{tikz}
\usetikzlibrary{intersections}
\usetikzlibrary{decorations.markings}
\usetikzlibrary{arrows}
\usetikzlibrary{positioning}

\usepackage{enumitem}
\usepackage{cleveref} 

\usepackage[T1]{fontenc}
\DeclareMathAlphabet{\mathpzc}{OT1}{pzc}{m}{it}
\usepackage{yfonts}
\usepackage[sans]{dsfont}
\usepackage{txfonts}
\usepackage[mathscr]{euscript}
\usepackage{bbm}

\newtheorem{thm}{Theorem}
\newtheorem{cor}[thm]{Corollary}
\newtheorem{lem}[thm]{Lemma}

\theoremstyle{definition}		
\newtheorem{defn}[thm]{Definition}	

\newcommand{\ring}[1]{\mathcal{#1}}	

\newcommand{\nat}{\mathbb N}			

\newcommand{\ff}[1]{{\mathbb F}_{#1}}		
\newcommand{\ffs}[1]{{\mathbb F}_{#1}^\star}	
\newcommand{\ffx}[1]{\ff{#1}[X]}		
\newcommand{\ffxy}[1]{\ff{#1}[X,Y]}		
\newcommand{\ffxyz}[1]{\ff{#1}[X,Y,Z]}		

\DeclareMathOperator{\Tr}{Tr}

\def\nonsquare{\boxslash}

\newcommand{\N}{\ring N}

\newcommand{\GC}{\ring T}

\begin{document}

\title[A PTR polynomial for the Hughes planes]
{A PTR polynomial for the Hughes planes and a new class of permutation
polynomials involving Catalan numbers}

\author[S. Brittain]{Stephen Brittain}
\author[R.S. Coulter]{Robert S. Coulter}
\address[S. Brittain \& R.S. Coulter]{Department of Mathematical Sciences, University of Delaware,
Newark, DE, 19716, United States of America.}

\author[A.M.W. Hui]{Alice Man Wa Hui}
\address[A.M.W. Hui]{School of Mathematics and Statistics, Clemson University,
Clemson, SC, 29634, United States of America.}

\begin{abstract}
Hughes introduced the projective planes that bear his name in 1957 and they have since been
studied extensively. However, until now, no polynomial representation of a
planar ternary ring that represents them has been determined. In this paper,
we rectify this omission by determining a reduced PTR polynomial for any Hughes
plane defined over a regular nearfield.
The polynomials obtained provide a new surprising connection: both the
Catalan numbers and generalized Catalan numbers occur among the coefficients,
depending on the representation.
Since every PTR polynomial has connections with several classes of permutation
polynomials, we obtain three new infinite classes of permutation polynomials as
a consequence of our main result, and these, too, involve the Catalan numbers.
The differential uniformity of new permutation polynomials is also determined.
\end{abstract}

\maketitle

\section{Introduction}

Throughout $q=p^e$ for some odd prime $p$ and $e\in\nat$ and
$Q=q^2$. We use $\ff{q}$ to
denote the finite field on $q$ elements, $\ffs{q}$ the non-zero elements,
and $\ffx{q}, \ffxy{q}, \ffxyz{q}$ to denote
the rings of polynomials in 1, 2 and 3 indeterminates over $\ff{q}$,
respectively.
The squares and non-squares of $\ff{q}$ are denoted by $\square_q$ and
$\nonsquare_q$, respectively.
The quadratic character will be denoted by $\eta$. That is, $\eta(x)=1$ if 
$x$ is a non-zero square, $\eta(x)=-1$ if $x$ is a non-square, and $\eta(0)=0$.
When the underlying field is $\ff{q}$, it is easily seen that
$\eta(x)=x^{(q-1)/2}$.

It is well known that any function on $\ff{q}$ can be represented by infinitely
many polynomials in $\ffx{q}$, and uniquely in {\em reduced form} over
$\ffx{q}$ by a polynomial of degree less than $q$. 
These facts extend to multivariate functions also.
In this paper we will be especially interested in a special type of trivariate
function on $\ff{q}$.
\begin{defn}
A function $T:\ff{q}^3\rightarrow\ff{q}$ is called a
{\em planar ternary ring (PTR)} if it satisfies the following 5 properties:
\begin{enumerate}[label=(\Alph*)]
\item $T(a,0,z)=T(0,b,z)=z$ for all $a,b,z\in\ff{q}$.
\item $T(x,1,0)=x$ and $T(1,y,0)=y$ for all $x,y\in\ff{q}$.
\item If $a,b,c,d\in\ff{q}$ with $a\ne c$, then there exists a unique
$x$ satisfying $T(x,a,b)=T(x,c,d)$.
\item If $a,b,c\in\ff{q}$, then there is a unique $z$ satisfying
$T(a,b,z)=c$.
\item If $a,b,c,d\in\ff{q}$ with $a\ne c$, then there is a unique pair
$(y,z)$ satisfying $T(a,y,z)=b$ and $T(c,y,z)=d$.
\end{enumerate}
\end{defn}
The motivation for studying such functions stems from work of Hall
\cite{H-1943-pp}, who introduced the concept of a PTR defined over an arbitrary
set when he developed a method for introducing a coordinate system on an
arbitrary projective plane. In fact, Hall proved in
\cite{H-1943-pp} that a PTR $T$ defined on an arbitrary set was equivalent to a
projective plane. Moreover, the only specific requirement of the arbitrary
set was that it had cardinality equal to the order of the plane.

It is well known that the only finite projective planes so far found have
prime-power order. Consequently, to study the known finite planes, there is
nothing prohibiting the underlying coordinatizing
set being chosen to be a finite field of the appropriate order.
Since we have the polynomial representation of functions over finite fields
mentioned above, a polynomial $T\in\ffxyz{q}$ that induces a PTR under
evaluation is referred to as a {\em PTR polynomial}. This idea was formalized
and expounded upon by Coulter in \cite{C-2019-ocpop}, and we refer the reader
to that paper for more information about general forms of PTR polynomials.

PTR polynomials have been determined for a number of different
planes. For instance they are known for the classical plane, as well as a host
of commutative semifield planes: in the odd order case through their
equivalence with planar DO polynomials, see Coulter and Henderson
\cite{CH-2008-cpas}, and Kosick \cite{K-2009-csoop}; in the even order case
see Kantor and Williams \cite{KW-2003-sspal}.
While studying the concept of optimal coordinatization, Joshi
\cite{J-2023-asopp} determined PTR polynomials for a number of projective
planes of small order, including all of the known projective
planes of order 16.
Before the work of this article, the only non-translation planes for which a
PTR polynomial had been determined were:
\begin{itemize}
\item Four Lenz-Barlotti type I.1 planes of order 16; see \cite{M-2017-pposo}.
(For information about the Lenz-Barlotti classification, see Dembowski
\cite{D-1968-fg}, Section 3.1.)
\item The Mathon plane of order 16; see \cite{M-2017-pposo}.
(This is the only Lenz-Barlotti type II.1 plane of order 16 known.)
\item The Hughes planes of order 9 and 25.
\item The Figeuroa plane of order 27, \cite{F-1982-afont}.
\item The Lenz-Barlotti type II.1 planes of Coulter and Matthews
\cite{CM-1997-pfapo}.
\end{itemize}
The PTR polynomials for the planes of small order just listed were determined
by Joshi \cite{J-2023-asopp}, while the PTR polynomials for the
Coulter-Matthews planes can be obtained directly from
the planar polynomials used to construct them in \cite{CM-1997-pfapo}.
There are two known infinite classes of Lenz-Barlotti type I planes --
the Hughes planes and the Figueroa planes -- and strikingly, 
no general form of a PTR polynomial for either of them 
has been determined until now, despite these two classes of planes
being among the most important examples known.

In this article, we partially rectify this by determining a PTR polynomial for
the Hughes planes. Specifically, we shall prove the following.
\begin{thm}\label{hughesptr}
Let $T\in\ffxyz{Q}$ be the polynomial defined by
\begin{equation*}
T(X,Y,Z)
= M(X,Y) + Z - \sum_{i=0}^{q-2}g_i(X)t_q(Y)^{i+1}t_q(Z)^{q-(i+1)},
\end{equation*}
where
\begin{equation*}
M(X,Y)=XY-2^{-1}t_{(Q+1)/2}(X)t_q(Y)
\end{equation*}
and
\begin{equation*}
g_i(X) = (-4)^{-(i+1)}\sum_{j=0}^{i+1}  C[j(q-1)+i]X^{j(q-1)+i+1}.
\end{equation*}
Here, $C[n]$ is the $n$-th Catalan number and $t_n(X)=X^n-X$.
Then $T$ is a reduced PTR polynomial for the Hughes plane of order $Q$ defined over
the regular nearfield.
\end{thm}
Coulter showed in \cite{C-2019-ocpop}, Theorem 4.2, that there are a number of
permutation polynomial (PP) classes represented by a single PTR polynomial.
In particular, for \Cref{hughesptr}, we have
\begin{itemize}
\item for each $(y,z)\in\ffs{Q}\times\ff{Q}$, $T(X,y,z)$ is a PP over $\ff{Q}$.
\item for each $(x,z)\in\ffs{Q}\times\ff{Q}$, $T(x,Y,z)$ is a PP over $\ff{Q}$.
\item for each $(x,y)\in\ff{Q}\times\ff{Q}$, $T(x,y,Z)$ is a PP over $\ff{Q}$.
\end{itemize}
Thus, \Cref{hughesptr} produces three infinite classes of PPs over $\ff{Q}$ for
arbitrary $Q=p^{2e}$, $p$ an odd prime and $e\in\nat$. To the best of our
knowledge, these are the first PP classes found where the Catalan numbers
appear as coefficients. An alternative version of the PTR polynomial, see
\Cref{generalizedcoeffs}, shows that these same PPs can be written in such
a way as to have the generalized Catalan numbers appearing as coefficients.

The paper is layed out as follows. In \Cref{hughesplanes} we give a brief
description of the Hughes planes and introduce a functional form of the PTR that
has been known for many years. \Cref{involution} introduces an involutory
PP which lies at the heart of our PTR polynomial construction. The following
section deals with some preliminary results, mostly regarding identities
involving the Catalan numbers. The PTR polynomial is determined in
\Cref{ptrpoly}, establishing \Cref{hughesptr} as well as the alternative
identity where the generalized Catalan numbers appear. For completeness, in the
final section we consider the differential uniformity of the PPs arising from
\Cref{hughesptr}.

\section{The Hughes planes} \label{hughesplanes} 

The Hughes planes were introduced in 1957 by Hughes \cite{H-1957-acond}.
They were the first infinite class of projective planes found without any
point-line transitivities. His construction also contained
the only previously known example of such a plane, a plane of order 9
discovered 50 years earlier by Veblen and Wedderburn \cite{VW-1907-ndanp}.
At the time of writing, the Hughes planes and the planes of Figueroa
\cite{F-1982-afont} (from 1982) remain the only known infinite classes of
projective planes
with this property. As such, they have been studied extensively.

The Hughes plane can be defined over any nearfield of order $Q=q^2$ with
center $\ff{q}$. Nearfields were first identified and studied by Dickson
\cite{D-1905-doaga, D-1905-ofa}. Dickson constructed an infinite class, now
known as the regular nearfields, along with 7 sporadic examples, now known as
the irregular nearfields.
Zassenhaus \cite{Z-1935-uef} proved in 1935 that Dickson had, in fact,
determined all nearfields in his original work.
In this paper, we will only be interested in the Hughes planes constructed from
regular nearfields, as this represents an infinite class of projective planes.

\subsection{The regular nearfields}

The regular nearfield of order $Q=q^2$ with center $\ff{q}$ is the algebraic
structure $\N=(\ff{Q},+,\star)$ where $+$ is
field addition in $\ff{Q}$ and $\star$ is defined by
\begin{equation*}
x\star y=
\begin{cases}
xy &\text{if $x\in\square_Q$,}\\
xy^q &\text{if $x\in\nonsquare_Q$.}
\end{cases}
\end{equation*}
We note that the definition for nearfield multiplication in the regular
nearfield is usually given with the Frobenius map on the
$x$ with the square/nonsquare condition on the $y$. We use this equivalent form
for convenience only. Our final PTR polynomial description will be the same, but
for the $X$ and $Y$ terms switched. We take care in all that follows to be
consistent -- the observant reader will see that the PTR function given below
also has the indeterminates $x$ and $y$ switched when compared directly with
the source text.

\subsection{The PTR function for the Hughes planes}

Dembowski \cite{D-1968-fg}, page 248, provides the following description of
a PTR for the Hughes plane: for $x,y,z\in\ff{Q}$,
\begin{equation*}
T(x,y,z)=
\begin{cases}
x\star y+z &\text{if $y\in\ff{q}$,}\\
(x+k)\star y+k' &\text{if $y\notin\ff{q}$, where $(k,k')\in\ff{q}\times\ff{q}$
uniquely satisfies $z=ky+k'$}.
\end{cases} 
\end{equation*}
Ostensibly, our task is straightforward. We want to determine a polynomial
version of the function $T$. The difficulty lies in understanding the
conditions that describe the two cases and providing a polynomial that
encompasses these conditions along with the behavior of the function.

To this end, we first note that if $x+k\in\square_Q$, then
$$(x+k)\star y+k'=xy+ky+k'=xy+z,$$
and if $x+k\in\nonsquare_Q$, then
\begin{align*}
    (x+k)\star y+k' = (x+k)y^q+k'
    &=xy^q+ky^q+k'\\
    &= xy^q+(ky+k')^q\\
    &= xy^q+z^q.
\end{align*}
We can therefore rewrite $T$ as
\begin{equation}\label{piecewiseT}
T(x,y,z)=
\begin{cases}
xy+z &\text{ if $y\in \ff{q}$,}\\
xy+z &\text{ if $y\notin\ff{q}$ and $x+k\in\square_Q$,
where $(k,k')\in\ff{q}\times\ff{q}$ uniquely satisfies $z=ky+k'$,}\\
xy^q+z^q &\text{ if $y\notin\ff{q}$ and $x+k\in\nonsquare_Q$,
where $(k,k')\in\ff{q}\times\ff{q}$ uniquely satisfies $z=ky+k'$.}
\end{cases}
\end{equation}

\section{An involutory permutation polynomial} \label{involution}

We now introduce a permutation polynomial over $\ff{Q}$ that is fundamental to
our PTR polynomial construction.
For any $k\in\ff{Q}$, define $\phi_k\in\ffx{Q}$ by
$\phi_k(X)=(X+k)^{(Q+1)/2}-k$.
The key properties of $\phi_k$ are outlined in the following lemma.
\begin{lem} \label{philemma}
The following statements hold.
\begin{enumerate}[label=(\roman*)]
\item $\phi_k$ is a permutation polynomial over $\ff{Q}$.
\item $\phi_k$ induces an involution under evaluation. Specifically, for
$x\in\ff{Q}$, 
\begin{equation*}
\phi_k(x) =
\begin{cases}
x	&\text{if $x+k\in\square_Q$,}\\
-x-2k	&\text{if $x+k\in\nonsquare_Q$.}
\end{cases} 
\end{equation*}
\end{enumerate}
\end{lem}
\begin{proof}
It suffices for us to prove (ii), since this implies (i).
Note that
$$\phi_k(x)=(x+k)^{(Q-1)/2}(x+k)-k=\eta(x+k)\, (x+k)-k.$$ 
Clearly, $\phi_k(-k)=-k$. For $x\ne-k$, we have
\begin{equation*}
\phi_k(x) =
\begin{cases}
(x+k) -k=x&\text{if $x+k\in\square_Q$,}\\
-(x+k)-k=-x-2k	&\text{if $x+k\in\nonsquare_Q$.}
\end{cases} 
\end{equation*}
Thus, we see $\phi_k$ induces the claimed function. It remains to prove this
function is an involution.
Since $\phi_k$ acts as the identity map on the set $\{x\,:\,x+k\in\square_Q\}$,
it is enough to observe that if $x+k\in\nonsquare_Q$, then
\begin{align*}
\phi_k(\phi_k(x)) &=\phi_k(-x-2k)\\
&=-(-x-2k)-2k \qquad\text{(as $-1\in\square_Q$)}\\
&=x,
\end{align*}
which proves (ii).
\end{proof}

\section{Some identities involving binomial coefficients and the Catalan numbers}

There are a number of facts needed to prove our main results. Some of these 
are concerned with binomial coefficients and the Catalan numbers.
We shall outline those now.

For a fixed prime $p$, the base-$p$ expansion of a number $\alpha$
will be represented as $\alpha=(\alpha_k\cdots\alpha_0)_p$, where
$$\alpha = \sum_{i=0}^k \alpha_i p^i, \text{ with } 0\le \alpha_i<p.$$
We shall be dealing with binomial coefficients throughout, including fractional
binomial coefficients. We use the convention 
$\binom{n}{k} = 0$ when $n<k$ and $n$ is an integer.
We recall the following result of Lucas \cite{L-1878-tdfns}.
\begin{lem}[Lucas' Theorem] \label{lucas}
Let $p$ be a prime and $\alpha\ge\beta$ be positive integers with $\alpha$
and $\beta$ having base-$p$ expansions
$\alpha=(\alpha_t\cdots\alpha_0)_p$ and $\beta=(\beta_t\cdots\beta_0)_p$,
respectively.
Then
\begin{equation*}
\binom{\alpha}{\beta}\equiv \prod_{i=0}^t \binom{\alpha_i}{\beta_i} \bmod p.
\end{equation*}
\end{lem}
We immediately apply Lucas' Theorem to prove the following identity.
\begin{lem}\label{lemma from lucas}
Let $0\le a,b<q$ be integers. Then
$$\binom{(Q+1)/2}{aq+b}\equiv\binom{(q-1)/2}{a}\binom{(q+1)/2}{b}\bmod p.$$
In particular, if $\binom{(Q+1)/2}{aq+b}\not\equiv 0\bmod p$, then
$a\le (q-1)/2$ and $b\le (q+1)/2$.
\end{lem}
\begin{proof}
Let $a=(a_{e-1}\cdots a_0)_p$ and $b=(b_{e-1}\cdots b_0)_p$.
Then
$$aq+b=\sum_{k=0}^{e-1}b_kp^k+\sum_{k=e}^{2e-1}a_{k-e}p^k.$$
As
\begin{equation*}
\frac{p^m+1}{2}=\frac{p+1}{2}+\sum_{k=1}^{m-1}\frac{p-1}{2}p^k,
\end{equation*}
for any $m\in\nat$, we may use Lemma \ref{lucas} for $m=2e$ and $m=e$ to obtain
\begin{align*}
\binom{(Q+1)/2}{aq+b}&\equiv \binom{\frac{p+1}{2}}{b_0}\prod_{k=1}^{e-1}\binom{\frac{p-1}{2}}{b_k}\prod_{k=0}^{e-1}\binom{\frac{p-1}{2}}{a_k}\bmod p\\
&\equiv \binom{\frac{p+1}{2}+\sum_{k=1}^{e-1}\frac{p-1}{2}p^k}{b}\binom{\sum_{k=0}^{e-1}\frac{p-1}{2}p^k}{a}\bmod p\\
&\equiv \binom{\frac{q+1}{2}}{b}\binom{\frac{q-1}{2}}{a}\bmod p, 
\end{align*}
which proves both claims.
\end{proof}

We shall also make use of the following easily observed result concerning
fractional binomial coefficients modulo a prime.
\begin{lem} \label{fractionalcoeffs}
Let $n,t$ be integers so that $0\le n < p^t$. Then 
\begin{equation*}
\binom{\pm 1/2}{n}\equiv \binom{(p^t\pm 1)/2}{n}\bmod p.
\end{equation*}
\end{lem}
To see this, it is enough to note that those numbers divisible by $p$ in the
expansion of the left hand side are divisible by the same power of $p$ as the
corresponding number in the expansion of the right hand side.
\begin{lem} \label{fractionalidentity}
Let $n,t$ be integers so that $1\le n < p^t$. Then 
\begin{equation*}
 2\binom{2n-1}{n} \equiv (-4)^n\binom{(p^t-1)/2}{n}\bmod p.
\end{equation*}
\end{lem}
\begin{proof}
Note that
\begin{align*}
2 \binom{2n-1}{n}
&= \frac{2 (2n-1)!}{(n-1)!n!}\\
&= \frac{2^n (n-1)! (1)(3)\cdots(2n-1)}{(n-1)!n!}\\
&= \frac{(-4)^n\left(\frac{-1}{2}\right)^n (1)(3)\cdots(2n-1)}{n!}\\
&= \frac{(-4)^n
\left(\frac{-1}{2}\right)\left(\frac{-3}{2}\right)\cdots
\left(\frac{-1}{2}-(n-1)\right)}{n!}\\
&=(-4)^n \binom{-1/2}{n}.
\end{align*}
By Lemma \ref{fractionalcoeffs},
\begin{equation*}
\binom{-1/2}{n} (-4)^n \equiv \binom{(p^t-1)/2}{n} \bmod p.
\end{equation*}
The proof is complete.
\end{proof}
\begin{defn}
Let $n,k\ge 0$ be integers.
\begin{enumerate}[label=(\roman*)]
\item The {\em $n$th Catalan number} $C[n]$ is defined by
$$C[n] = \frac{1}{n+1}\binom{2n}{n}.$$
\item The {\em generalized Catalan number $\GC[n,k]$} is defined by
\begin{equation*}
\GC[n,k]=\binom{2n}{n}\binom{2k}{k}\frac{2k+1}{n+k+1}.
\end{equation*}
\end{enumerate}
\end{defn}
We adopt the convention of $\binom{n}{k}=0$ if $n<k$ or $n<0$ or $k<0$.
The following identity can be derived by considering
the generating function of the Catalan numbers.
\begin{lem}\label{Catalan identity}
Let $n,t$ be integers so that $0\le n<p^t-1$. Then
\begin{equation*}
C[n] \equiv  2 (-4)^n \binom{(p^t+1)/2}{n+1}\bmod p.
\end{equation*}
\end{lem}
\begin{proof}The proof is similar to that of Lemma \ref{fractionalidentity}.
\begin{align*}
C[n]=\frac{1}{n+1}\binom{2n}{n}
=&\frac{1}{n+1}\frac{(1)(3)\cdots(2n-1)}{n!}\frac{(2)(4)(6)\cdots(2n)}{n!}\\
=&\frac{2^n}{n+1}\frac{(1)(3)\cdots(2n-1)}{n!}\\
=&\frac{4^n}{n+1}\frac{(\frac12)(\frac32)\cdots(\frac{2n-1}{2})}{n!}\\
=&(2)(-4)^n\frac{(\frac12)(\frac{-1}{2})(\frac{-3}{2})\cdots(\frac{1}{2}-n)}{(n+1)!}\\
=&2(-4)^n\binom{1/2}{n+1}.
\end{align*}
The result now follows from Lemma \ref{fractionalcoeffs} as $n+1<p^t$.
\end{proof}
We will be interested in an equivalence modulo $p$ between Catalan
numbers and generalized Catalan numbers. To this end, we first prove an
identity.
\begin{lem} \label{catidentity}
For $n\ge 0$ and $k\ge 1$, we have
\begin{equation*}
\GC[n,k] - \GC[n+1,k-1] = 2 \binom{2k-1}{k} C[n].
\end{equation*}
\end{lem}
\begin{proof}
The proof is standard algebraic manipulation. We have
\begin{align*}
\GC[n,k] - \GC[n+1,k-1]
&= \binom{2n}{n}\binom{2k}{k} \frac{2k+1}{n+k+1}
-  \binom{2n+2}{n+1}\binom{2k-2}{k-1} \frac{2k-1}{n+k+1}\\
&=\frac{1}{n+k+1}\binom{2n}{n}\binom{2k-2}{k-1}
\left(\frac{(2k+1)2k(2k-1)}{k^2} - \frac{(2n+2)(2n+1)(2k-1)}{(n+1)^2}\right)\\
&= \frac{2(2k-1)}{n+k+1}\binom{2n}{n}\binom{2k-2}{k-1}
\left(\frac{2k+1}{k} - \frac{2n+1}{n+1}\right)\\
&= \frac{2(2k-1)}{n+k+1}\binom{2n}{n}\binom{2k-2}{k-1}
\left(\frac{n+k+1}{k(n+1)}\right)\\
&= 2 \binom{2n}{n}\binom{2k-1}{k}\frac{1}{n+1},
\end{align*}
which proves the claim.
\end{proof}
\begin{lem}\label{relating Catalan to gen. Catalan}
Let $0\le n<q-1$ and $0\le k < q$. Then
$$C[kq+n]\equiv\GC[n,k]-\GC[n+1,k-1]\bmod p.$$
\end{lem}
\begin{proof} If $k=0$, then
\begin{align*}
\GC[n,0]-\GC[n+1,-1]&=\GC[n,0]=\binom{2n}{n}\frac{1}{n+1}=C[n].
\end{align*}
For the remainder, suppose $k\ge 1$. 
In light of Lemma \ref{catidentity}, we need to prove
\begin{equation*}
C[kq+n]\equiv 2 \binom{2k-1}{k} C[n]\bmod p.
\end{equation*}
Appealing to Lemma \ref{Catalan identity} with $t=2e$ and $t=e$, and
Lemma \ref{lemma from lucas}, we have
\begin{align*}
C[kq+n] - 2 \binom{2k-1}{k} C[n]
&\equiv 2(-4)^{kq+n} \binom{(Q+1)/2}{kq+n+1} 
- 4 (-4)^n \binom{2k-1}{k}\binom{(q+1)/2}{n+1}\bmod p\\
&\equiv 2 (-4)^n\binom{(q+1)/2}{n+1}
\left((-4)^k\binom{(q-1)/2}{k} - 2\binom{2k-1}{k}\right)\bmod p\\
&\equiv 2 (-4)^n\binom{(q+1)/2}{n+1} \, N\bmod p,
\end{align*}
where
\begin{equation*}
N = (-4)^k\binom{(q-1)/2}{k} - 2\binom{2k-1}{k}.
\end{equation*}
As $1\le k<q$, Lemma \ref{fractionalidentity} now shows $N\equiv 0\bmod p$,
proving the lemma.
\end{proof}

\begin{lem}\label{catalan 0 condition}
Let $i,j$ be integers with $0\le i < i+1 < j\le (q-1)/2$.
Then $C[j(q-1)+i]\equiv 0\bmod p$.
\end{lem}
\begin{proof}
Since $p$ is odd, Lemma \ref{Catalan identity} shows $C[n]\equiv 0\bmod p$ if and
only if $\binom{(Q+1)/2}{n+1}\equiv 0 \bmod p$.
Therefore, it is sufficient to show $\binom{(Q+1)/2}{j(q-1)+i+1}\equiv 0\bmod p$.
Now,
\begin{align*}
\binom{(Q+1)/2}{j(q-1)+i+1} &=\binom{(Q+1)/2}{(j-1)q+q+1+i-j}\\
&\equiv \binom{(q-1)/2}{j-1}\binom{(q+1)/2}{q+1+i-j}\bmod p,
\end{align*}
by Lemma \ref{lemma from lucas}.
Under our hypotheses, $-(q-1)/2\le i-j \le -2$, so that
$(q+3)/2\le q+1+i-j\le q-1$. Thus,
$$\binom{(q+1)/2}{q+1+i-j}\equiv 0\bmod p$$
by Lemma \ref{lucas}, which proves the statement.
\end{proof}

\section{Determining the PTR polynomial} \label{ptrpoly}

The key to writing the PTR function $T$ as a polynomial $T\in\ffxyz{Q}$ is in
recognizing the role of the polynomial $\phi_k$ in the function's behavior.
Two other polynomials will also play an important role.
The first is the polynomial $t_n(X)=X^n-X$ for $n\in\nat$.
This polynomial, especially the
case $n=q$, has been shown to play an important role in the description of
PTR polynomials for semifield planes, see Coulter, Henderson and
Kosick \cite{CHK-2007-ppfcs}, Theorem 3.1 in particular.
It also appears in the PTR polynomials
for the order 16 planes that Joshi determined in \cite{J-2023-asopp}.
Usually, the existence of the polynomials $t_q(X), t_q(Y), t_q(Z)$ in the
PTR polynomial of a projective plane of order $q^n$ for some $n\in\nat$ will
point to the existence of a classical subplane of order $q$, though this
depends on how well the coordinatisation that generated the PTR has been
carried out.
Consequently, given the Hughes plane contains a classical subplane of order
$q$, it is no surprise it appears again in our results.
The second polynomial is the well known trace map from $\ff{Q}$ to $\ff{q}$:
$\Tr(X)=X^q+X$.

Before we begin working on the PTR polynomial, we first prove an identity
regarding the polynomial $t_q(X)$.
\begin{lem}\label{reduce t_q}
Let $i,j$ be non-negative integers with $i>0$.
We have $t_q(X)^{j(q-1)+i}\equiv (-1)^j t_q(X)^i\bmod (X^Q-X)$.
\end{lem}
\begin{proof}
Observe that $t_q(X)^q=X^Q-X^q\equiv -t_q(X)\bmod (X^Q-X)$.
Clearly, if $x\in\ff{q}$, then $t_q(x)=0$. For all
$x\in\ff{Q}\setminus\ff{q}$, we have
\begin{align*}
t_q(x)^{j(q-1)+i}&= \left(t_q(x)^q \right)^jt_q(x)^{i-j}\\
&= (-1)^jt_q(x)^jt_q(x)^{i-j}\\
&=(-1)^jt_q(x)^i.
\end{align*}
Thus, for all $x\in\ff{Q}$, we have $t_q(x)^{j(q-1)+i} = (-1)^j t_q(x)^i$, from
which the polynomial identity follows.
\end{proof}

We now move to determine a PTR polynomial for the Hughes planes. We begin by
producing a functional relation between the PTR and $\phi_k$.
\begin{lem} \label{functionform}
For all $x,y,z\in\ff{Q}$, we have
\begin{equation} \label{sigmaT}
T(x,y,z) = z + 2^{-1}\left(x\Tr(y) - t_q(y)\sigma(x,y,z)\right),
\end{equation}
where
\begin{equation*}
\sigma(x,y,z) = 
\begin{cases}
0 & \text{if $y\in \ff{q}$,} \\
\phi_k(x) &\text{if $y\notin\ff{q}$, with $k=t_q(z)/t_q(y)$.}
\end{cases} 
\end{equation*}
\end{lem}
\begin{proof}
The proof involves simply checking that the values of the form given coincide
with \Cref{piecewiseT}. Let $x,y,z\in\ff{Q}$ and set
$f:\ff{Q}^3\rightarrow\ff{Q}$ to be the function
$$f(x,y,z)= z + 2^{-1}\left(x\Tr(y) - t_q(y)\sigma(x,y,z)\right).$$

If $y\in \ff{q}$, then $t_q(y)=0$ and $\Tr(y)=2y$, and we have
$f(x,y,z)=z+xy$, as required.

For the remainder, suppose $y\notin \ff{q}$ and write $z=ky+k'$ with
$k,k'\in\ff{q}$. We note that
$$t_q(z)=t_q(ky+k')=t_q(ky) + t_q(k') = k\,t_q(y) +0 = k\, t_q(y).$$
As $y\notin\ff{q}$, $t_q(y)\ne 0$, and so $k=t_q(z)/t_q(y)$.
We thus have
\begin{equation*}
f(x,y,z) = z + 2^{-1}\left(x\Tr(y) - t_q(y)\sigma(x,y,z)\right)
 = z + 2^{-1}\left(x\Tr(y) - t_q(y)\phi_k(x)\right).
\end{equation*}
If $x+k$ is square, then $\phi_k(x)=x$ and we find
\begin{equation*}
f(x,y,z)
= z + 2^{-1}\left(x(y^q+y) - (y^q-y)x\right)
= z + xy = T(x,y,z).
\end{equation*}
If $x+k$ is non-square, then $\phi_k(x)=-x-2k$ and 
\begin{align*}
f(x,y,z) &= z + 2^{-1}\left(x\Tr(y) - t_q(y)\phi_k(x)\right)\\
&= z + 2^{-1}\left(x\Tr(y) + t_q(y)
\left(x + 2\frac{t_q(z)}{t_q(y)}\right)\right)\\
&= z + 2^{-1}\left(x(y^q+y) + x(y^q-y) + 2(z^q-z)\right)\\
&= z + xy^q + z^q -z\\
&=xy^q + z^q\\
&=T(x,y,z),
\end{align*}
completing the proof.
\end{proof}
Thanks to Lemma \ref{functionform}, since we know polynomial forms for $\Tr$ and
$t_q$, to produce a non-reduced polynomial form of the PTR polynomial, we need
a polynomial form of the function $\sigma$. 
\begin{lem} \label{phiklem}
A polynomial $\sigma\in\ffxyz{Q}$ which evaluates as the function $\sigma$
given in Lemma \ref{functionform} is
\begin{equation*}
\sigma(X,Y,Z)= t_q(Y)^{Q-1}
\left(X^{(Q+1)/2} +
\sum_{m=1}^{(Q-1)/2}\binom{(Q+1)/2}{m}X^m t_q(Y)^{m-1}t_q(Z)^{Q-m}\right).
\end{equation*}
\end{lem}
\begin{proof}
We show the claimed polynomial acts as required under evaluation.

Firstly, if $y\in\ff{q}$, then $t_q(y)=0$, and $\sigma(x,y,z)=0$ is clear.
Now suppose $y\notin\ff{q}$ and set $k=t_q(z)/t_q(y)$.
Note that $k\in\ff{q}$, so that $k^{(Q-1)/2}=1$ and $k^{(Q+1)/2}=k$.
We have
\begin{align*}
\sigma(x,y,z) &= t_q(y)^{Q-1}
\left(x^{(Q+1)/2} +
\sum_{m=1}^{(Q-1)/2}\binom{(Q+1)/2}{m}x^m t_q(y)^{m-1}t_q(z)^{Q-m}\right)\\
&= x^{(Q+1)/2} +
\sum_{m=1}^{(Q-1)/2}\binom{(Q+1)/2}{m}x^m t_q(y)^{m-1}t_q(z)^{1-m}\\
&= x^{(Q+1)/2} +
\sum_{m=1}^{(Q-1)/2}\binom{(Q+1)/2}{m}x^m  k^{-(m-1)}\\
&= x^{(Q+1)/2} +
\sum_{m=1}^{(Q-1)/2}\binom{(Q+1)/2}{m}x^m  k^{(Q+1)/2 - m}\\
&= \left(\sum_{m=0}^{(Q+1)/2}\binom{(Q+1)/2}{m}x^m  k^{(Q+1)/2 - m}\right) - k\\
&=\phi_k(x),
\end{align*}
as required. 
\end{proof}
We can now establish a non-reduced PTR polynomial for the Hughes planes.
\begin{lem}\label{nonreducedPTRpoly}
A non-reduced PTR polynomial for the Hughes plane of order $Q$ over the regular
nearfield is given by 
\begin{equation} \label{nonreducedpoly}
T(X,Y,Z)=M(X,Y)+Z
- 2^{-1} \sum_{m=1}^{(Q-1)/2} \binom{(Q+1)/2}{m}X^m t_q(Y)^m t_q(Z)^{Q-m},
\end{equation}
where $M(X,Y)=XY- 2^{-1} t_{(Q+1)/2}(X) t_q(Y)$.
\end{lem}
\begin{proof}
A combination of the previous two lemmas yields the polynomial form 
\begin{equation} \label{longform}
Z + 2^{-1}\left(X\Tr(Y)
- t_q(Y)^Q
\left(X^{(Q+1)/2} +
\sum_{m=1}^{(Q-1)/2}\binom{(Q+1)/2}{m}X^m t_q(Y)^{m-1}t_q(Z)^{Q-m}\right)
\right).
\end{equation}
Since $t_q(Y)^Q\bmod (Y^Q-Y) = t_q(Y)$, one now easily verifies that the
claimed polynomial $T(X,Y,Z)$ comes directly from \Cref{longform} after
replacing $t_q(Y)^Q$ with $t_q(Y)$.
\end{proof}

\subsection{Proof of \Cref{hughesptr}}
To prove \Cref{hughesptr} it remains to reduce the polynomial of
\Cref{nonreducedpoly}  modulo $X^Q-X, Y^Q-Y$ and $Z^Q-Z$.
Since $M(X,Y)$ is already reduced, this means we only need to determine the
reduced form of the sum in \Cref{nonreducedpoly}, which we'll denote by 
$S$. The $X$ terms of this sum are also clearly reduced, so we need to
consider the $Y$ and $Z$ terms.
First, by re-indexing Lemma \ref{nonreducedPTRpoly}, we have 
\begin{equation*}
S= \sum_{m=0}^{(Q-3)/2}\binom{(Q+1)/2}{m+1}X^{m+1}t_q(Y)^{m+1}t_q(Z)^{Q-m-1}.
\end{equation*}
For $0\le m\le (Q-3)/2$, we can write $m=j(q-1)+i$ for unique $i,j$ with
$0\le i\le q-2$ and $0\le j \le (q-1)/2$. Therefore,
\begin{equation*}
S = \sum_{i=0}^{q-2} \sum_{j=0}^{(q-1)/2}
\binom{(Q+1)/2}{j(q-1)+i+1}X^{j(q-1)+i+1}t_q(Y)^{j(q-1)+i+1}t_q(Z)^{Q-(j(q-1)+i)-1}.
\end{equation*}
One checks that $Q-(j(q-1)+i)-1=(q-j)(q-1)+q-i-1$.
Since $i+1>0$ and $q-i-1>0$, we may apply Lemma \ref{reduce t_q} to find
\begin{align*}
t_q(Y)^{j(q-1)+i+1}&\equiv (-1)^jt_q(Y)^{i+1} \bmod (Y^Q-Y), \text{ and}\\
t_q(Z)^{(q-j)(q-1)+q-i-1}&\equiv (-1)^{q-j}t_q(Z)^{q-i-1}\bmod (Z^Q-Z).
\end{align*}
Thus, 
\begin{equation*}
S\bmod (Y^Q-Y,Z^Q-Z) = \sum_{i=0}^{q-2}\sum_{j=0}^{(q-1)/2}
\binom{(Q+1)/2}{j(q-1)+i+1}X^{j(q-1)+i+1}t_q(Y)^{i+1}t_q(Z)^{q-i-1}.
\end{equation*}
To complete the reduction and establish \Cref{hughesptr}, it remains to
consider the coefficients in $2^{-1}S$.

By Lemma \ref{Catalan identity},
$-2^{-1} \binom{(Q+1)/2}{j(q-1)+i+1}
\equiv (-4)^{-(j(q-1)+i+1)}C[j(q-1)+i]\bmod p$.
Since $C[j(q-1)+i]\equiv 0\bmod p$ for $i+1<j<(q-1)/2$ (by
Lemma \ref{catalan 0 condition}) and 
\begin{align*}
(-4)^{j(q-1)+i+1} &= \left((-4)^{q-1}\right)^j(-4)^{i+1}\\
&\equiv(-4)^{i+1}\bmod p,
\end{align*}
\Cref{hughesptr} is established.

We note, in passing, that we have no geometric understanding as to why the
Catalan numbers might appear in a description of the Hughes planes.
We know of no previous result regarding the
Hughes planes, or (for that matter) any other projective plane, where the
Catalan numbers have been shown to be directly connected to the plane in any
way.

\subsection{A version of \Cref{hughesptr} involving generalized Catalan numbers}

While it is easy to see from the form of \Cref{hughesptr} how to extract a
$t_q(Y)$ or $t_q(Z)$ factor from the summation, factoring out a $t_q(X)$ factor
is more difficult. However, doing so reveals an alternate form of the reduced
polynomial where the coefficients are generalized Catalan numbers.
\begin{thm} \label{generalizedcoeffs}
Set $Q=q^2$ and let $T_2\in\ffxyz{Q}$ be the polynomial defined by
\begin{equation*}
T_2(X,Y,Z)
= M(X,Y) + Z + t_q(X)t_q(Y)t_q(Z)\sum_{i=0}^{q-2}h_i(X)t_q(Y)^{i}t_q(Z)^{q-i},
\end{equation*}
where
\begin{equation*}
M(X,Y)=XY-2^{-1}t_{(Q+1)/2}(X)t_q(Y)
\end{equation*}
and
\begin{equation*}
h_i(X) = (-4)^{-(i+1)}\sum_{j=0}^{i}  \GC[i-j,j]X^{j(q-1)+i}.
\end{equation*}
Then $T_2=T$, where $T$ is the polynomial of \Cref{hughesptr}.
\end{thm}
\begin{proof}
It suffices to show $g_i(X)=-t_q(X)h_i(X)$ for all $0\le i \le q-2$.
Now,
\begin{align*}
-(-4)^{i+1}t_q(X)h_i(X) &= (-4)^{i+1} \left(Xh_i(X) - X^q h_i(X)\right)\\
&=\sum_{j=0}^{i} \GC[i-j,j]X^{jq+i-j+1}
-\sum_{j=0}^{i} \GC[i-j,j]X^{(j+1)q+i-j}\\
&= \sum_{j=0}^{i} \GC[i-j,j]X^{jq+i-j+1}
-\sum_{j=1}^{i+1}  \GC[i-j+1,j-1]X^{jq+i-j+1},
\end{align*}
where in the last step we've simply shifted the index of the second sum by 1.
Note that in the second sum, if $j=0$, then we would have $\GC[i+1,-1]=0$.
Therefore,
\begin{equation*}
-(-4)^{i+1}t_q(X)h_i(X) = 
- \GC[0,i] X^{(i+1)q}
+ \sum_{j=0}^i \left( \GC[i-j,j]-\GC[i-j+1,j-1]\right) X^{jq+i-j+1}.
\end{equation*}
Applying Lemma \ref{relating Catalan to gen. Catalan} with $n=i-j$ and $k=j$ in the
sum, we find
\begin{equation*}
-(-4)^{i+1}t_q(X)h_i(X) = 
-\GC[0,i] X^{(i+1)q} + \sum_{j=0}^{i} C[jq+i-j] X^{jq+i-j+1}.
\end{equation*}
Finally, we note that 
\begin{align*}
-\GC[0,i] &= - \binom{2i}{i}\frac{2i+1}{i+1}\\
&= - \binom{2i+1}{i+1}\\
&= -2^{-1} 2 \binom{2i+1}{i+1}\\
&\equiv -2^{-1} (-4)^{i+1} \binom{(q-1)/2}{i+1}\bmod p
\qquad\text{(by Lemma \ref{fractionalidentity})} \\
&\equiv 2 (-4)^i \binom{(Q+1)/2}{(i+1)q}\bmod p
\qquad\text{(by Lemma \ref{lemma from lucas})} \\
&\equiv 2 (-4)^{qi} (-4)^{q-1} \binom{(Q+1)/2}{iq+q-1 + 1}\bmod p.
\end{align*}
Since $iq+q-1 \le (q-2)q+q-1 < Q-1$, we can appeal to Lemma \ref{Catalan identity}
to obtain
$$2 (-4)^{qi} (-4)^{q-1} \binom{(Q+1)/2}{iq+q-1 + 1}
\equiv C[iq+q-1]\bmod p.$$
Thus,
\begin{align*}
-(-4)^{i+1}t_q(X)h_i(X) &= 
-\GC[0,i] X^{(i+1)q} + \sum_{j=0}^{i} C[jq+i-j] X^{jq+i-j+1}\\
&= C[iq+q-1] X^{(i+1)q} + \sum_{j=0}^{i} C[jq+i-j] X^{jq+i-j+1}\\
&= \sum_{j=0}^{i+1} C[jq+i-j] X^{jq+i-j+1}.
\end{align*}
Therefore $-t_q(X)h_i(X)=g_i(X)$, as required.
\end{proof}

\section{Differential uniformity considerations} \label{dusection}

We end this article by considering the
the differential uniformity of the PPs arising from our PTR polynomial.
\begin{defn}
Let $f\in\ffx{q}$.
\begin{enumerate}[label=(\roman*)]
\item The \emph{uniformity of $f$}, denoted $u(f)$, is defined by
\[u(f):=\max_{b\in\ff{q}} \#\{a\in\ff{q}\,:\, f(a)=b\}.\]

\item For $a\in\ffs{q}$, the {\em differential operator of $f$ in the
direction of $a$} is defined by
\begin{equation*}
\Delta_{f,a}(x):=f(x+a)-f(x).
\end{equation*}
\item The \emph{differential uniformity (DU) of $f$} is defined by 
\begin{equation*}
\delta_{f}:=\max_{a\in\ffs{q}} u(\Delta_{f,a}).  
\end{equation*}
\end{enumerate}
\end{defn}
This concept was introduced by Nyberg \cite{N-1993-dumic} in a more general
setting.
She showed that the lower the DU of a function, the more resistance it offers
against differential cryptanalysis when used in a substitution box.
It is known that
those functions with the optimal DU of 1 cannot be bijections, see Coulter
and Senger \cite{CS-2014-otnod, CS-xxxx-motno},
Most desirable, then, are bijective functions with low, but not optimal DU,
and so we are motivated to consider the DU of the PPs arising from
\Cref{hughesptr}.
Our motivation for examining the DU of the PPs arising from our PTR
polynomial is motivated by results of Bergman and Coulter \cite{BC-2022-cfwld},
who constructed reasonably low DU functions from regular nearfield planes.
Unfortunately, we shall show that the bijective maps arising from the Hughes
planes have quite high DU.

Recall that a polynomial $L\in\ffx{q}$ is a {\em linearized (or additive)}
polynomial if $L(X)=\sum_i c_i X^{p^i}$. Linearized polynomials satisfy
$L(x+y)=L(x)+L(y)$ for all $x,y\in\ff{q}$. Their relevance here is outlined in
the following lemma. We omit the almost trivial proofs.
\begin{lem} \label{simplifying}
Let $f,L\in\ffx{q}$ with $L$ linearized and let $c\in\ff{q}$.
The following statements hold.
\begin{enumerate}[label=(\roman*)]
\item $\delta_{f+L+c}=\delta_f$.
\item If $L$ is a PP, then $\delta_{f(L)}=\delta_{L(f)}=\delta_f$.
\item If $L$ is not a PP, then $\Delta_{f(L),a}(X)$ is a constant polynomial
for any root $a$ of $L$. In particular, $\delta_{f(L)}=q$.
\end{enumerate}
\end{lem}
\begin{cor}
The differential uniformity of each of $T(x,Y,z)$ and $T(x,y,Z)$ over $\ff{Q}$
is $Q$ (maximal).
\end{cor}
\begin{proof}
For $x,y,z$, the polynomials $T(x,Y,z)$ and $T(x,y,Z)$ can be seen to consist
of a constant term, a linear term, and a sum involving a composition with
$t_q(Y)$ or $t_q(Z)$. In either case, Lemma \ref{simplifying} (i) and (ii) reduces
the problem to considering the DU of the sum only, and since each version
involves a composition with a linearized polynomial that is not a PP, the DU is
seen to be maximal by Lemma \ref{simplifying} (iii).
\end{proof}
This leaves us to consider the DU of $T(X,y,z)$.
For any $y,z\in\ff{Q}$, we define $f_{y,z}(X)=T(X,y,z)$.
\begin{lem} \label{duX}
The following statements hold.
\begin{enumerate}[label=(\roman*)]
\item If $y\in\ff{q}$, then $\delta_{f_{y,z}}=Q$.

\item If $y\notin\ff{q}$, then $\delta_{f_{y,z}}=(Q+3)/4$.
\end{enumerate}
\end{lem}
\begin{proof}
If $y\in\ff{q}$, then $f_{y,z}(X)=yX+z$.
Since every difference operator is a constant for a linear polynomial,
it is clear the DU is $Q$ in this case.

For $y\notin\ff{q}$, set $k=t_q(z)/t_q(y)$. It follows from
Lemma \ref{functionform} that $f_{y,z}(X)$ is equivalent under evaluation to the
polynomial $z+2^{-1} \Tr(y) X - t_q(y) \phi_k(X)$.
Since we may ignore constant and linearized terms, we see
$\delta_{f_{y,z}}=\delta_{\phi_k}$.
Applying Lemma \ref{simplifying} (i) and (ii) now reduces us to determining the DU of
$f(X)=X^{(Q+1)/2}=X\eta(X)$.
Clearly, we have
\begin{equation*}
\Delta_{f,a}(x) =
\begin{cases}
a	&\text{if $x+a\in\square_Q$ and $x\in\square_Q$,}\\
2x+a	&\text{if $x+a\in\square_Q$ and $x\in\nonsquare_Q$,}\\
-2x-a	&\text{if $x+a\in\nonsquare_Q$ and $x\in\square_Q$,}\\
-a	&\text{if $x+a\in\nonsquare_Q$ and $x\in\nonsquare_Q$.}
\end{cases}
\end{equation*}
We'll refer to these cases as Case 1 through 4.
Clearly the number of pre-images of an image will be maximized when we're in
Case 1 or Case 4, and (possibly) with an additional one value coming from
one of Case 3 or Case 2, respectively.
If Case 3 added an additional pre-image for $a$, then $x=-a$,
but then we would also need $x+a=0\notin\nonsquare_Q$, so this case doesn't
occur.
If Case 2 yields an additional pre-image for $-a$, then $x=a$. This would also
require $2a\in\square_Q$. Since $2\in\square_Q$, this would require
$a\in\square_Q$, implying $x\in\square_Q$, a contradiction.
Thus, the DU of $f$ is given by the maximum number of pre-images for $a$ or
$-a$, or put another way, the largest of the cardinalities over all
$a\in\ffs{Q}$ of the sets $K_1(a)$ and $K_4(a)$ given by
\begin{align*}
K_1(a) &=\{ x\in\ff{Q}\,:\, x\in\square_Q \land x+a\in\square_Q\},\\
K_4(a) &=\{ x\in\ff{Q}\,:\, x\in\nonsquare_Q \land x+a\in\nonsquare_Q\}.
\end{align*}
Bergman, Coulter and Fain \cite{BCF-2025-ocran} determined these sets
explicitly. An examination of their results shows the largest cardinality is
$(Q+3)/4$ for $K_1(a)$ when $a\in\square_Q$.
Thus, $\delta_f=(Q+3)/4$, as claimed.
\end{proof}

\bibliographystyle{amsplain}

\providecommand{\bysame}{\leavevmode\hbox to3em{\hrulefill}\thinspace}
\providecommand{\MR}{\relax\ifhmode\unskip\space\fi MR }
\providecommand{\MRhref}[2]{%
  \href{http://www.ams.org/mathscinet-getitem?mr=#1}{#2}
}
\providecommand{\href}[2]{#2}

\end{document}